\documentclass[a4paper]{amsart}
\usepackage{amssymb, enumitem}
\usepackage{hyperref, aliascnt}

\newcommand{\andSep}{\,\,\,\text{ and }\,\,\,}
\newcommand{\CC}{{\mathbb{C}}}
\newcommand{\NN}{{\mathbb{N}}}

\newcommand{\ca}{$C^*$-algebra}
\newcommand{\stHom}{$*$-ho\-mo\-mor\-phism}
\newcommand{\calI}{\mathcal{I}}
\newcommand{\calL}{\mathcal{L}}
\newcommand{\calO}{\mathcal{O}}
\newcommand{\Bdd}{\mathcal{B}}

\DeclareMathOperator{\linSpan}{span}
\DeclareMathOperator{\Nil}{Nil}
\DeclareMathOperator{\FN}{FN}
\DeclareMathOperator{\N}{N}
\DeclareMathOperator{\Cu}{Cu}

\def\today{\number\day\space\ifcase\month\or   January\or February\or
   March\or April\or May\or June\or   July\or August\or September\or
   October\or November\or December\fi\   \number\year}

\newtheorem{lma}{Lemma}[section]

\newaliascnt{thmCt}{lma}
\newtheorem{thm}[thmCt]{Theorem}
\aliascntresetthe{thmCt}

\newaliascnt{corCt}{lma}
\newtheorem{cor}[corCt]{Corollary}
\aliascntresetthe{corCt}

\newaliascnt{prpCt}{lma}
\newtheorem{prp}[prpCt]{Proposition}
\aliascntresetthe{prpCt}

\theoremstyle{definition}

\newaliascnt{dfnCt}{lma}
\newtheorem{dfn}[dfnCt]{Definition}
\aliascntresetthe{dfnCt}

\newaliascnt{rmkCt}{lma}
\newtheorem{rmk}[rmkCt]{Remark}
\aliascntresetthe{rmkCt}

\newaliascnt{exaCt}{lma}
\newtheorem{exa}[exaCt]{Example}
\aliascntresetthe{exaCt}

\newaliascnt{qstCt}{lma}
\newtheorem{qst}[qstCt]{Question}
\aliascntresetthe{qstCt}

\newaliascnt{pgrCt}{lma}
\newtheorem{pgr}[pgrCt]{}
\aliascntresetthe{pgrCt}

\newcounter{theoremintro}

\newaliascnt{thmIntroCt}{theoremintro}
\newtheorem{thmIntro}[thmIntroCt]{Theorem}
\aliascntresetthe{thmIntroCt}

\newaliascnt{pbmIntroCt}{theoremintro}
\newtheorem{pbmIntro}[pbmIntroCt]{Problem}
\aliascntresetthe{pbmIntroCt}

\title{Lie ideals in properly infinite C*-algebras}

\author{Hannes Thiel}
\address{Hannes~Thiel, 
Department of Mathematical Sciences, Chalmers University of Technology and University of
Gothenburg, Gothenburg SE-412 96, Sweden.}
\email{hannes.thiel@chalmers.se}
\urladdr{www.hannesthiel.org}

\thanks{
The first named author was partially supported by the Knut and Alice Wallenberg Foundation (KAW 2021.0140).
}

\subjclass[2020]%
{Primary
46L05, 
46L10. 
Secondary
16W10, 
17B60, 
47B47. 
}
\keywords{Lie ideals, $C^*$-algebras, properly infinite, von Neumann algebras, commutators, square-zero elements}
\date{\today}

\begin{document}

\begin{abstract}
We show that every Lie ideal in a unital, properly infinite \ca{} is commutator equivalent to a unique two-sided ideal.
It follows that the Lie ideal structure of such a \ca{} is concisely encoded by its lattice of two-sided ideals.
This answers a question of Robert in this setting.

We obtain similar structure results for Lie ideals in unital, real rank zero \ca{s} without characters.
As an application, we show that every Lie ideal in a von Neumann algebra is related to a unique two-sided ideal, which solves a problem of Bre\v{s}ar, Kissin, and Shulman.
\end{abstract}

\maketitle

\section{Introduction}

A \emph{Lie ideal} in a \ca{} $A$ is a linear subspace $L \subseteq A$ such that for every $a \in A$ and $x \in L$ the commutator $[a,x] := ax-xa$ belongs to $L$.
The study of Lie ideals in \ca{s} has a long history, and we refer to Section~2 in \cite{Mar10ProjCommutLieIdls} for an overview of available results and directions of active research.

Immediate examples of Lie ideals in a \ca{}~$A$ are all subspaces contained in the center $Z(A)$ and all subspaces containing the commutator subspace $[A,A]$.
If~$A$ is unital and simple, then this does in fact describes all Lie ideals -- a consequence of Herstein's description of Lie ideals in simple rings \cite{Her55LieJordanSimpleRing}.
Since the center of a unital, simple \ca{} $A$ contains only scalar multiples of the identity, a subspace~$L$ is a Lie ideal in~$A$ if and only if $L=\{0\}$, $L=\CC 1$, or $[A,A] \subseteq L$.

Pop \cite{Pop02FiniteSumsCommutators} showed that a unital \ca{} satisfies $A=[A,A]$ if and only if $A$ has no tracial states.
Consequently, a simple, unital \ca{} $A$ without tracial states has exactly three Lie ideals: $\{0\}$, $\CC 1$ and $A$.
Cuntz and Pedersen \cite{CunPed79EquivTraces} gave a description of the closure $\overline{[A,A]}$ as the intersection of the kernels of all tracial states on $A$.
In particular, a simple, unital \ca{} $A$ with a unique tracial state has exactly four closed Lie ideals: $\{0\}$, $\CC 1$, $\overline{[A,A]}$ and $A$.
This was first observed by Marcoux and Murphy \cite{MarMur98UniInvSpaceCAlg}, generalizing the case of UHF-algebras, which was first proved by Marcoux \cite{Mar95CldLieIdlsCAlgs}.

Ng and Robert \cite{NgRob16CommutatorsPureCa} showed that $[A,A]$ is closed in every pure, exact \ca.
It follows that a simple, unital, exact, pure \ca{} $A$ with a unique tracial state has exactly four Lie ideals: $\{0\}$, $\CC 1$, $[A,A]$ and $A$.
This applies for example to UHF-algebras, to irrational rotation algebras, and to the Jiang-Su algebra (where it was first shown by Ng \cite{Ng12CommutatorsJiangSu}).
It also applies to many, and possibly all, simple, reduced group \ca{s}.
Indeed, if $G$ is a countable, exact, acylindrically hyperbolic group with trivial finite radical and with the rapid decay property, then $C^*_\mathrm{red}(G)$ has strict comparison \cite{AmrGaoElaPat24arX:StrCompRedGpCAlgs}, and therefore is pure \cite{AntPerThiVil24arX:PureCAlgs}, and has a unique tracial state \cite{BreKalKenOza17CSimpleUniqueTr}.

\medskip

Turning towards Lie ideals in non-simple \ca{s}, it is clear that the structure of the lattice of two-sided ideals will play a role.
A systematic study of the interplay between Lie and two-sided ideals was initiated by Bre\v{s}ar, Kissin, and Shulman \cite{BreKisShu08LieIdeals}
We first introduce the necessary notation.
Following the usual convention, given subspaces $M,N \subseteq A$, we use $[M,N]$ to denote the subspace of $A$ generated by the set of commutators $[x,y]$ for $x \in M$ and $y \in N$.
Further, given a subspace $K \subseteq A$, we set
\[
T(K) := \big\{ a \in A : [A,a] \subseteq K \big\}.
\]
This `derived subspace' was first considered by Herstein \cite{Her55LieJordanSimpleRing}, and it is denoted by $N(K)$ in \cite{BreKisShu08LieIdeals}.

Given a two-sided ideal $I \subseteq A$, Bre\v{s}ar, Kissin, and Shulman observed that a subspace $L \subseteq A$ is a Lie ideal whenever 
\[ 
[A,I] \subseteq L \subseteq T([A,I]).
\] 
We say that a Lie ideal $L$ is \emph{embraced by} $I$ if it satisfies these inclusions. 
The Lie ideals that are embraced by some two-sided ideal are thought of as `unsurprising'.
A major question is if there exist any Lie ideals that do not arise this way:

\begin{pbmIntro}[{Bre\v{s}ar, Kissin, and Shulman \cite[p.~74]{BreKisShu08LieIdeals}}]
\label{pbm:Embraced}
Given a \ca{} $A$, is every Lie ideal of $A$ embraced by some two-sided ideal of $A$?
\end{pbmIntro}

To study this problem, two further notions were introduced in \cite{BreKisShu08LieIdeals}:
Given a Lie ideal $L \subseteq A$ and a two-sided ideal $I \subseteq A$, we say that $L$ is \emph{commutator equivalent to} $I$ if
\[
[A,I] = [A,L].
\]
Further, we say that $L$ is \emph{related to} $I$ if 
\[
[A,I] \subseteq L \subseteq T(I).
\]

It is easy to see that these constitute a sufficient and a necessary conditions for embracement.
Indeed, if $L$ is commutator equivalent to~$I$, then $L$ is embraced by~$I$.
Further, if $L$ is embraced by $I$, then $L$ is related to $I$.

It was shown in \cite[Theorem~5.27]{BreKisShu08LieIdeals} that every Lie ideal $L$ in a \ca{}~$A$ is topologically commutator equivalent to a two-sided ideal $I$ in the sense that $[A,L]$ and $[A,I]$ have the same closure.
This leads to a description of all \emph{closed} Lie ideals in terms of \emph{closed} two-sided ideals.
However, this does not show that all closed Lie ideals in a \ca{} are commutator equivalent to (let alone embraced by or related to) a two-sided ideal.
For general non-closed Lie ideals even less is known, and \autoref{pbm:Embraced} as well as the following stronger problem remain open:

\begin{pbmIntro}[{Robert \cite[Question~1.12]{Rob16LieIdeals}}]
\label{pbm:Equivalent}
Given a \ca{} $A$, is every Lie ideal of $A$ commutator equivalent to some two-sided ideal of $A$?
\end{pbmIntro}

The results mentioned at the beginning of the introduction show that every Lie ideal in a unital, simple \ca{} $A$ is commutator equivalent to a two-sided ideal, which gives positive solutions to Problems~\ref{pbm:Embraced} and~\ref{pbm:Equivalent} in this setting.
Indeed, the only two-sided ideals are~$\{0\}$ and~$A$, and a Lie ideal is commutator equivalent to~$\{0\}$ if and only if it is contained in~$Z(A)$, while a Lie ideal it is commutator equivalent to~$A$ if and only if it contains~$[A,A]$. 
(This uses that $[A,A]=[A,[A,A]]$, which follows for example from \cite[Corollary~3.4]{GarThi24PrimeIdealsCAlg}.)

In \cite{FonMieSou82LieJordanIdlsBH}, Fong, Miers and Sourour showed that every Lie ideal in the von Neumann algebra of bounded, linear operators on a separable Hilbert space is embraced by a two-sided ideal.
This was substantially extended by Bre\v{s}ar, Kissin, and Shulman, who showed in \cite[Theorem~5.19]{BreKisShu08LieIdeals} that every Lie ideal in a von Neumann algebra is commutator equivalent to a two-sided ideal, thereby providing positive solutions to Problems~\ref{pbm:Embraced} and~\ref{pbm:Equivalent} for von Neumann algebras.

Our main result is a positive solution to \autoref{pbm:Equivalent} for unital, properly infinite \ca{s}. 

\begin{thmIntro}[{\ref{prp:PropInf}}]
\label{thm:PropInf}
Let $L$ be a Lie ideal in a unital, properly infinite \ca{}~$A$.
Then $L$ is commutator equivalent to the two-sided ideal $I:=A[A,L]A$, and hence also embraced by~$I$, and related to~$I$.
Further, $I$ is the only two-sided ideal of $A$ to which $L$ is related.
\end{thmIntro}

The proof of \autoref{thm:PropInf} is inspired by a result of Marcoux \cite{Mar95CldLieIdlsCAlgs}, who showed that if $B$ is a unital \ca{} and $n \geq 2$, then every Lie ideal in the matrix algebra $A=M_n(B)$ is related to a two-sided ideal in~$A$.
Our main technical advancement is a generalization of this result to unital \ca{s} $A$ that admit a unital $*$-homomorphism $M_2(\CC) \oplus M_3(\CC) \to A$.

The next result summarizes Theorems~\ref{prp:L-Related-IdlAL}, \ref{prp:CharLieEmbr},  \ref{prp:AAI-AI}, and~\ref{prp:CorrespondenceLie}.¨

\begin{thmIntro}
\label{prp:MapFromM2M3}
Let $A$ be a unital \ca{} that admits a unital $*$-homomorphism $M_2(\CC) \oplus M_3(\CC) \to A$.
The following statements hold:
\begin{enumerate}
\item
Given a Lie ideal $L \subseteq A$, for the two-sided ideal $I=A[A,L]A$, we have
\[
[A,I] = [A,[A,L]] \subseteq [A,L] \subseteq I, \andSep
[A,I] \subseteq L \subseteq T(I).
\]
Further, $I$ is the only two-sided ideal of $A$ to which $L$ is related.
\item
A Lie ideal $L \subseteq A$ is embraced by some two-sided ideal (which then necessarily is $A[A,L]A$) if and only if $L$ is commutator equivalent to some two-sided ideal, if and only if $[A,L]=[A,[A,L]]$.
\item
Given a two-sided ideal $I \subseteq A$, we have
\[
I = A[A,I]A, \andSep [A,I]=[A,[A,I]].
\]
\item
There is a natural bijection between the set $\calI$ of two-sided ideals in $A$ and the set 
\[
\calL = \big\{ L \subseteq A \text{ Lie ideal} : L = [A,L] \big\}.
\]
given by the maps $I \mapsto [A,I]$ and $L \mapsto A[A,L]A$.
\end{enumerate}
\end{thmIntro}

We note that \autoref{prp:MapFromM2M3} applies to all unital, properly infinite \ca{s}, which then yields \autoref{thm:PropInf}.
It also applies to unital, real rank zero \ca{s} without characters, and in particular to all von Neumann algebras with zero commutative summand, which we use to recover the solution to \autoref{pbm:Equivalent} for von Neumann algebras from \cite{BreKisShu08LieIdeals}.
Our approach also shows that a Lie ideal in a von Neumann algebra without commutative summand is related to a unique two-sided ideal, which solves \cite[Problem~5.21]{BreKisShu08LieIdeals}.

\begin{thmIntro}[{\ref{prp:VNA}}]
\label{thm:VNA}
Let $L$ be a Lie ideal in a von Neumann algebra~$M$.
Then $L$ is commutator equivalent to the two-sided ideal $I:=M[M,L]M$, and hence also embraced by~$I$, and related to~$I$.

If $M$ has zero commutative summand, then $I$ is the only two-sided ideal of $M$ to which $L$ is related.
\end{thmIntro}

\subsection*{Conventions}

Given a \ca{}, we use $A_+$ to denote its collection of positive elements, and we use $\widetilde{A}$ to denote the minimal unitization.
The absolute value of an element $x \in A$ is defined as $|x|:=(x^*x)^{\frac{1}{2}}$.

\subsection*{Acknowledgements}

The author thanks Leonel Robert for valuable comments regarding the Dixmier property for \ca{s}, and Matej Bre\v{s}ar for helpful feedback on an earlier version of this paper.
The author is also grateful to the anonymous referees for their insightful suggestions, in particular for pointing out a consequence of the results that is now included as \autoref{prp:ClosedLieIdl}.

\section{Lie ideals in C*-algebras with weakly divisible unit}
\label{sec:LieDivUnit}

In this section, we prove that every Lie ideal $L$ in a \ca{} $A$ that admits a unital \stHom{} $M_2(\CC) \oplus M_3(\CC) \to A$ is related to a unique two-sided ideal~$I$ in~$A$;
see \autoref{prp:L-Related-IdlAL}.
We further show that $L$ is commutator equivalent to $I$ if and only if $[A,L]=[A,[A,L]]$;
see \autoref{prp:CharLieEmbr}.

\medskip

Following \cite[Deﬁnition~5.1]{PerRor04AFembeddings}, a non-zero projection $p$ in a \ca{} is said to be \emph{weakly divisible of degree $n$} if there exists a unital $*$-homomorphism
\[
M_{n_1}(\CC) \oplus M_{n_2}(\CC) \oplus \ldots \oplus M_{n_r}(\CC) \to pAp
\]
for some natural numbers $r \geq 1$ and $n_1,\ldots,n_r \geq n$.
The $*$-homomorphism may additionally be assumed to be injective.

As noted in \cite[p.164]{PerRor04AFembeddings}, every matrix algebra $M_n(\CC)$ for $n \geq 2$ admits a unital (not necessarily injective) $*$-homomorphism $M_2(\CC) \oplus M_3(\CC) \to M_n(\CC)$, and it follows that a projection $p$ weakly divisible of degree $2$ if and only if there exists a unital (not necessarily injective) $*$-homomorphism $M_2(\CC) \oplus M_3(\CC) \to pAp$.
We deduce that the unit in a \ca{} $A$ is weakly divisible of degree $2$ if and only if there exists an injective $*$-homomorphism
\[
M_{2}(\CC) \to A, \quad
M_{3}(\CC) \to A, \quad\text{ or }\quad
M_{2}(\CC) \oplus M_3(\CC) \to A.
\]
In the first two cases, we get $A \cong M_2(B)$ and $A \cong M_3(C)$ for suitable unital \ca{s} $B$ and $C$, and for such matrix algebras it was shown by Marcoux \cite[Theorem~2.6]{Mar95CldLieIdlsCAlgs} that every Lie ideal is related to a two-sided ideal (which is also unique, as shown in \cite[Corollary~4.18]{BreKisShu08LieIdeals}).

Our main contribution to \autoref{prp:L-Related-IdlAL} is therefore in the case that $A$ admits a unital (injective) $*$-homomorphism $M_{2}(\CC) \oplus M_3(\CC) \to A$ but no unital $*$-homomorphism $M_n(\CC) \to A$ for $n=2$ or $n=3$ (or actually any $n \geq 2$).
The Cuntz algebra $\calO_\infty$ is an important such example.
(Using $K$-theory, we see that there exists no unital $*$-homomorphism $M_n(\CC) \to \calO_\infty$ for any $n \geq 2$.
On the other hand, $\calO_\infty$ admits a unital $*$-homomorphism from $M_{2}(\CC) \oplus M_3(\CC)$ as explained in the proof of \autoref{prp:PropInf}.)

\medskip

Our approach is inspired by the methods developed by Marcoux \cite{Mar95CldLieIdlsCAlgs} to describe Lie ideals in matrix algebras, as well as the techniques of Bre\v{s}ar, Kissin, and Shulman \cite[Section~3]{BreKisShu08LieIdeals} to study Lie ideals in algebras that decompose with respect to an idempotent.
However, the result does not simply follow by combining these methods, for example by applying \cite[Theorem~3.4]{BreKisShu08LieIdeals} for a suitable idempotent in the \ca.
Specifically, given a unital $*$-homomorphism $M_2(\CC) \oplus M_3(\CC) \to A$, we do not decompose $A$ with respect to the projection $p$ given by the image of the unit of $M_2(\CC)$, in which case $pAp+(1-p)A(1-p)$ is a direct sum of matrix algebras (whence its Lie ideals are related to two-sided ideals), but it is unclear if (3.17) in \cite[Theorem~3.4]{BreKisShu08LieIdeals} holds (or if $p$ is locally cyclic in the sense of \cite[Definition~3.5]{BreKisShu08LieIdeals}).
Instead, we use a `mixed' decomposition with respect to a projection $g$ that is given by the image of the sum of a rank-one projection in $M_2(\CC)$ with a rank-one projection in $M_3(\CC)$, as explained in \autoref{pgr:Setup}.
In this case, one can show that $g$ is locally cyclic, which is the idea underlying the proof that $A_{12}LA_{12} \subseteq [A,[A,L]]$ in \autoref{prp:InclusionsAAL}.
On the other hand, the algebra $gAg+(1-g)A(1-g)$ is not a sum of matrix algebras and it is therefore not immediate that its Lie ideals are related to two-sided ideals.

\begin{pgr}
\label{pgr:Setup}
Let $A$ be a unital \ca{} that admits a unital \stHom{}
\[
\varphi \colon M_2(\CC) \oplus M_3(\CC) \to A.
\]
Let $(e_{ij})_{i,j=1,2}$ denote the image of matrix units in $M_2(\CC)$ under $\varphi$, and similarly let $(f_{ij})_{i,j=1,2,3}$ denote the image of matrix units in $M_3(\CC)$ under $\varphi$.
We do not assume that $\varphi$ is injective, and therefore the $e_{ij}$ or $f_{ij}$ may be zero.

Set
\[
g := e_{11} + f_{11}, \andSep
h := e_{22} + f_{22} + f_{33}.
\]
Then $1=g+h$, and we view elements in $A$ as operator matrices with respect to the decomposition induces by $g$ and $h$.
More formally, we set
\[
A_{11} := gAg, \quad
A_{12} := gAh, \quad
A_{21} := hAg, \andSep
A_{22} := hAh,
\]
and every element $a$ in $A$ can be written uniquely as a sum $a=a_{11}+a_{12}+a_{21}+a_{22}$ with $a_{ij} \in A_{ij}$ for $i,j=1,2$.
\end{pgr}

\begin{lma}
\label{prp:InclusionsAAL}
Retain the notation from \autoref{pgr:Setup}, and let $L \subseteq A$ be a Lie ideal.
Then the sets 
\[
A_{11}LA_{22}, \quad
A_{22}LA_{11}, \quad
A_{12}LA_{12}, \andSep
A_{21}LA_{21}
\]
are all contained in $[A,[A,L]]$, and thus also in $[A,L]$ and in $L$.
\end{lma}
\begin{proof}
\emph{We verify that $A_{11}LA_{22} \subseteq [A,[A,L]]$.}
We first show that $gLh \subseteq L$.
To see this, let $x \in L$.
Then
\[
gx - xg
= [g, x] 
\in [A,L], \andSep
gx - 2gxg + xg
= [g,[g, x]] \in [A,[A,L]] 
\subseteq [A,L].
\]
Adding these expressions, we get
\[
2gxh
= 2gx - 2gxg
\in [A,L] \subseteq L,
\]
and using that $hg=0$ we get
\[
gxh
= [g,gxh]
= [\tfrac{1}{2}g,2gxh]
\in [A,L] \subseteq L.
\]
We have shown that $gLh \subseteq L$.

Now, given $x \in L$ and $a,b \in A$, we have
\[
gagxhbh
= [[gag,gxh],hbh]
\in [[A,L],A]
= [A,[A,L]].
\]
We have shown that $A_{11}LA_{22} \subseteq [A,[A,L]]$.

Analogously, one verifies that $A_{22}LA_{11} \subseteq [A,[A,L]]$.

\medskip

\emph{We verify that $A_{12}LA_{12} \subseteq [A,[A,L]]$.}
Given $v \in A_{12}$ and $y \in L$, we have
\begin{equation}
\label{eq:InclusionsAAL:vxv}
vyv 
= [-\tfrac{1}{2}v,[v,y]]
\in [A,[A,L]].
\end{equation}

Set
\[
v_1 := e_{12} + f_{12}, \quad
w_1 := e_{21} + f_{21}, \quad
v_2 := e_{12} + f_{13}, \andSep
w_2 := e_{21} + f_{31}.
\]
Then $v_1,v_2 \in A_{12}$ and $w_1,w_2 \in A_{21}$, and we have
\[
v_1w_1 
= v_2w_2 
= e_{11} + f_{11} 
= g, \quad
w_1v_1 
= e_{22} + f_{22}, \andSep
w_2v_2 
= e_{22} + f_{33}.
\]

Now, given $x \in L$ and $a,b \in A_{12}=gAh$, set
\[
b' = b(\tfrac{1}{2}e_{22}+f_{22}+f_{33}).
\]
Then $b'$ belongs to $gAh$ and we have 
\[
b = b'(2e_{22} + f_{22} + f_{33}).
\]

We have already proved that $hAhLgAg \subseteq L$, which gives
\[
w_1axb'w_1 \in L, \andSep
w_2axb'w_2 \in L.
\]

Applying \eqref{eq:InclusionsAAL:vxv} for $y=w_1axb'w_1$ and $v=v_1$, we get
\[
axb'(e_{22} + f_{22})
= v_1(w_1axb'w_1)v_1 \in [A,[A,L]].
\]
Analogously, we obtain that
\[
axb'(e_{22} + f_{33})
= v_2(w_2axb'w_2)v_2 \in [A,[A,L]],
\]
and consequently
\[
axb = axb'(2e_{22} + f_{22} + f_{33}) \in [A,[A,L]].
\]
We have shown that $A_{12}LA_{12} \subseteq [A,[A,L]]$.

Analogously, one verifies that $A_{21}LA_{21} \subseteq [A,[A,L]]$.
\end{proof}

\begin{lma}
\label{prp:AgLhA-AhLgA}
Retain the notation from \autoref{pgr:Setup}, and let $L \subseteq A$ be a Lie ideal.
Then 
\[
AgLhA = AhLgA.
\]
\end{lma}
\begin{proof}
Using at the first step that $A=AgA=AhA$ since $g$ and $h$ are full, and applying at the third step that $hAgLhAg \subseteq L$ by \autoref{prp:InclusionsAAL}, we get
\[
AgLhA
= AhAgLhAgA
= Ah(hAgLhAg)gA
\subseteq AhLgA.
\]
Analogously, we have
\[
AhLgA
= AgAhLgAhA
= Ag(gAhLgAh)hA
\subseteq AgLhA,
\]
as desired.
\end{proof}

\begin{lma}
\label{prp:AL-in-I}
Retain the notation from \autoref{pgr:Setup}, and let $L \subseteq A$ be a Lie ideal.
Then 
\[
[A,L] \subseteq AgLhA.
\]
\end{lma}
\begin{proof}
To simplify notation, we set $I = AgLhA$.
We have
\[
[A,gLh] \subseteq AgLh + gLhA \subseteq I.
\]
Further, using \autoref{prp:AgLhA-AhLgA} at the third step, we have
\[
[A,hLg] \subseteq AhLg + hLgA \subseteq AhLgA = AgLhA = I.
\]

Thus, given $x \in L$, we have $[A,gxh+hxg] \subseteq I$.
Note that $gxh, hxg \in L$ by \autoref{prp:InclusionsAAL}.
It follows that
\[
gxg + hxh = x - gxh - hxg \in L.
\]

We need to show that $[A,gxg + hxh] \subseteq I$.
Using additivity of the Lie bracket and that $I$ is an additive subgroup, it suffices to verify that $[A_{ij},gxg + hxh] \subseteq I$ for $i,j=1,2$.

\medskip

\emph{We show that $[A_{12},gxg + hxh] \subseteq I$.}
Given $a \in A$, we have
\[
[gah,gxg+hxh] 
= gahxh - gxgah.
\]
Using that $gxg + hxh \in L$, we have $[gah,gxg+hxh] \in L$, and thus
\[
[gah,gxg+hxh] 
\in L \cap gAh 
= gLh 
\subseteq AgLhA 
= I.
\]
We have shown that $[A_{12},gxg + hxh] \subseteq I$.

Similarly, and using also \autoref{prp:AgLhA-AhLgA} at the last step, we get
\[
[hag,gxg+hxh] 
= hagxg - hxhag \in L \cap hAg = hLg \subseteq AhLgA = I
\]
for every $a \in A$.
This shows that $[A_{21},gxg + hxh] \subseteq I$.

\medskip

\emph{We show that $[A_{11},gxg + hxh] \subseteq I$.}
Given $a \in A$, set $y = gagxg - gxgag$.
Then 
\[
y
= gagxg - gxgag
= [gag,gxg+hxh] 
\in [A,L] \subseteq L.
\]
Then
\[
y(e_{12}+f_{12})
= [y,e_{12}+f_{12}] 
\in [L,A] \subseteq L
\]
and thus
\[
y(e_{12}+f_{12})
\in L \cap gAh 
= gLh \subseteq I.
\]
We deduce that
\[
y 
= y(e_{12}+f_{12})(e_{21}+f_{21})
\in IA = I.
\]

\emph{We show that $[A_{22},gxg + hxh] \subseteq I$.}
Given $b \in A$, set $z = hbhxh - hxhbh$.
Then 
\[
z
= hbhxh - hxhbh
= [hbh,gxg+hxh] 
\in L \cap A_{22}.
\]
Then
\[
(e_{12}+f_{12})z
= [e_{12}+f_{12},z] \in L \cap gAh \subseteq I, \andSep
f_{13}z = [f_{13},z] \in I.
\]
We deduce that
\[
z = (e_{21}+f_{21})(e_{12}+f_{12})z + f_{31}f_{13}z \in AI+AI = I,
\]
as desired.
\end{proof}

\begin{prp}
\label{prp:CharIdlAL}
Retain the notation from \autoref{pgr:Setup}, and let $L \subseteq A$ be a Lie ideal.
Then 
\[
A[A,L] = A[A,L]A = AgLhA = AhLgA.
\]
\end{prp}
\begin{proof}
In general, if $M \subseteq A$ is a Lie ideal, then $AMA=AM$;
see, for example, \cite[Lemma~3.2]{GarLeeThi24arX:FullyNoncentral}.
Since $[A,L]$ is a Lie ideal, we get $A[A,L]=A[A,L]A$.

By \autoref{prp:AL-in-I}, we have $[A,L] \subseteq AgLhA$, and thus $A[A,L]A \subseteq AgLhA$.
By \autoref{prp:AgLhA-AhLgA}, we have $AgLhA = AhLgA$.
Finally, by \autoref{prp:InclusionsAAL}, we have $gLh \subseteq [A,L]$, and thus $AgLhA \subseteq A[A,L]A$.
\end{proof}

\begin{lma}
\label{prp:AI-in-AAL}
Retain the notation from \autoref{pgr:Setup}, and let $L \subseteq A$ be a Lie ideal, and set $I = A[A,L]A$.
Then 
\[
[A,I] \subseteq [A,[A,L]].
\]
\end{lma}
\begin{proof}
For each $i,j=1,2$, set $I_{ij} := I \cap A_{ij}$.
Since $I$ is a two-sided ideal, we have $I=I_{11}+I_{12}+I_{21}+I_{22}$.
By \autoref{prp:CharIdlAL}, we have $I = AgLhA$, and thus
\[
I 
= AgLhA
= \underbrace{gAgLhAg}_{I_{11}} 
+ \underbrace{gAgLhAh}_{I_{12}} 
+ \underbrace{hAgLhAg}_{I_{21}}
+ \underbrace{hAgLhAh}_{I_{22}}.
\]
By additivity of the Lie bracket and since $[A,L]$ is an additive subgroup, it suffices to show that $[A,I_{ij}] \subseteq [A,[A,L]]$ for each $i,j=1,2$.

By \autoref{prp:InclusionsAAL}, we have 
\[
I_{12} = gAgLhAh \subseteq [A,L], \andSep
I_{21} = hAgLhAg \subseteq [A,L],
\]
and thus $[A,I_{12}] \subseteq [A,[A,L]]$ and $[A,I_{21}] \subseteq [A,[A,L]]$.

\medskip

\emph{We show that $[A,I_{11}] \subseteq [A,[A,L]]$.}
It suffices to verify that $[A_{ij},I_{11}] \subseteq [A,[A,L]]$ for each $i,j=1,2$.
Using that $hg=0$, and using \autoref{prp:InclusionsAAL} at the last step, we get
\[
[A_{12},I_{11}]
= [gAh,gAgLhAg]
= -(gAgLhAg)(gAh)
= gAgLhAgAh
\subseteq [A,[A,L]].
\]
Similarly, we have
\[
[A_{21},I_{11}]
= [hAg,gAgLhAg]
= (hAg)(gAgLhAg)
= hAgAgLhAg
\subseteq [A,[A,L]].
\]
We also have $[A_{22},I_{11}]=\{0\}\subseteq [A,[A,L]]$.

Further, using \autoref{prp:InclusionsAAL} at the last step, we have
\[
I_{11} 
= (gAgLh)Ag
\subseteq [gAgLh,Ag]+Ag(gAgLh)
\subseteq [A,L] + AgLh.
\]
Using again \autoref{prp:InclusionsAAL} at the last step, we get
\begin{align*}
[A_{11},I_{11}]
&\subseteq [A_{11},[A,L]] + [A_{11},AgLh]
\subseteq [A,[A,L]] +[gAg,AgLh] \\
&\subseteq [A,[A,L]] +gAgAgLh 
\subseteq [A,[A,L]].
\end{align*}
We have verified that $[A,I_{11}] \subseteq [A,[A,L]]$.

Analogously, one shows that $[A,I_{22}] \subseteq [A,[A,L]]$, which completes the proof.
\end{proof}

The next result verifies statement~(1) in \autoref{prp:MapFromM2M3}.

\begin{thm}
\label{prp:L-Related-IdlAL}
Let $A$ be a unital \ca{} that admits a unital \stHom{} $M_2(\CC) \oplus M_3(\CC) \to A$, and let $L \subseteq A$ be a (not necessarily closed) Lie ideal.
Then, for the two-sided ideal $I:=A[A,L]A$, we have
\begin{equation}
\label{eq:L-Related-IdlAL}
[A,I] = [A,[A,L]] \subseteq [A,L] \subseteq I, \andSep
[A,I] \subseteq L \subseteq T(I).
\end{equation}
Further, $I$ is the only two-sided ideal of $A$ to which $L$ is related.
\end{thm}
\begin{proof}
Using \autoref{prp:AL-in-I} at the first step, and \autoref{prp:CharIdlAL} at the second step, we have
\[
[A,L] \subseteq AgLhA = I.
\]
This implies that $[A,[A,L]] \subseteq [A,I]$, and the converse inclusion $[A,I] \subseteq [A,[A,L]]$ is shown in \autoref{prp:AI-in-AAL}.
Using that $L$ is a Lie ideal, we have $[A,L] \subseteq L$, and therefore $[A,[A,L]] \subseteq [A,L]$.
This shows the left chain of inclusions in \eqref{eq:L-Related-IdlAL}.

The inclusion $[A,I] \subseteq L$ follows using that $[A,I] \subseteq [A,L]$ and $[A,L] \subseteq L$.
Finally, since $[A,L] \subseteq I$, we have $L \subseteq T(I)$.

To show uniqueness of $I$, let $J$ be a two-sided ideal that is related to $L$, that is, such that
\[
[A,J] \subseteq L \subseteq T(J).
\]
Then $[A,L] \subseteq J$, and since $I$ is the two-sided ideal generated by $[A,L]$ we get $I \subseteq J$.

Applying \autoref{prp:InclusionsAAL} at the first step for $J$, considered as a Lie ideal, we get
\[
gJh \subseteq [A,J] \subseteq L,
\]
and thus
\[
gJh \subseteq L \cap gAh = gLh.
\]
Since $g$ and $h$ are full, we have $A=AgA$ and $A=AhA$, and consequently
\[
J
= AJA
= AgAJAhA
= AgJhA
\subseteq AgLhA
= I.
\]
In conclusion, we get $I=J$, as desired.
\end{proof}

\begin{cor}
\label{prp:ClosedLieIdl}
Let $A$ be a unital \ca{} that admits a unital \stHom{} $M_2(\CC) \oplus M_3(\CC) \to A$, and let $L \subseteq A$ be a closed Lie ideal.
Then, the two-sided ideal $I:=A[A,L]A$ is closed.
\end{cor}
\begin{proof}
By \autoref{prp:L-Related-IdlAL}, $L$ is related to $I$, that is, we have $[A,I] \subseteq L \subseteq T(I)$.
Using that~$L$ is closed, it follows that the closure $\overline{I}$ satisfies $[A,\overline{I}] \subseteq L \subseteq T(\overline{I})$.
Since~$I$ is the only two-sided ideal of $A$ to which $L$ is related (\autoref{prp:L-Related-IdlAL}), we get $I=\overline{I}$.
\end{proof}

The next result verifies statement~(2) in \autoref{prp:MapFromM2M3}.

\begin{thm}
\label{prp:CharLieEmbr}
Let $A$ be a unital \ca{} that admits a unital \stHom{} $M_2(\CC) \oplus M_3(\CC) \to A$, and let $L \subseteq A$ be a (not necessarily closed) Lie ideal.
Then the following statements are equivalent:
\begin{enumerate}
\item
The Lie ideal $L$ is commutator equivalent to the two-sided ideal $A[A,L]A$.
\item
The Lie ideal $L$ is embraced by some two-sided ideal.
\item
We have $[A,L] = [A,[A,L]]$.
\end{enumerate}
\end{thm}
\begin{proof}
It is clear that~(1) implies~(2).
Assuming~(2), let us show that~(3) holds.
Set $I := A[A,L]A$, and let $J \subseteq A$ be some two-sided ideal that embraces $L$.
Then~$L$ is also related to $J$.
It follows that $I=J$, since $I$ is the unique two-sided ideal to which $L$ is related, by \autoref{prp:L-Related-IdlAL}.
It follows that $L \subseteq T([A,I])$.
Using this at the first step, and using \autoref{prp:L-Related-IdlAL} at the second step, we get
\[
[A,L] \subseteq [A,I] = [A,[A,L]],
\]
which  verifies~(3).

Assuming~(3), let us show that~(1) holds.
Using that assumption at the firs step, and \autoref{prp:L-Related-IdlAL} at the second step, we get
\[
[A,L] = [A,[A,L]] = [A,I],
\]
which verifies~(1).
\end{proof}

The next result is essentially contained in \cite[Proposition~5.7]{PerRor04AFembeddings}.
We only show how to remove the separability assumption.
A \emph{character} on a \ca{} $A$ is one-dimensional, irreducible representation, that is, a surjective \stHom{} $A \to \CC$.

\begin{prp}[Perera, R{\o}rdam]
\label{prp:RR0}
Let $A$ be a unital \ca{} of real rank zero.
Then $A$ has no characters if and only if $A$ admits a unital \stHom{} $M_2(\CC) \oplus M_3(\CC) \to A$.
\end{prp}
\begin{proof}
The backwards implication is clear.
To show the forward implication, assume that $A$ has no characters.
By \cite[Lemma~3.5]{KirRor15CentralSeqCharacters}, there exists a \emph{separable} sub-\ca{} $A_0 \subseteq A$ containing the unit of~$A$.
Using that real rank zero is separably inheritable (see \cite[Section~II.8.5]{Bla06OpAlgs}), we obtain a \emph{separable} sub-\ca{} $B \subseteq A$ of real rank zero with $A_0 \subseteq B$.
It follows that $B$ has no characters, which allows us to apply \cite[Proposition~5.7]{PerRor04AFembeddings} for $B$.
We obtain a unital \stHom{} from $M_2(\CC) \oplus M_3(\CC)$ to $B$, and thus also to~$A$.
\end{proof}

\begin{exa}
\label{exa:RR0}
Let $A$ be a unital \ca{} of real rank zero that has no characters.
Then \autoref{prp:L-Related-IdlAL} applies to $A$.
In particular, every Lie ideal in $A$ is related to a (unique) two-sided ideal.
\end{exa}

We expect that the following question has a positive answer.

\begin{qst}
\label{qst:NoChar}
Let $A$ be a unital \ca{} that has no characters.
Is every Lie ideal in $A$ is related to a (unique) two-sided ideal?
\end{qst}

\section{Two-sided ideals in C*-algebras with weakly divisible unit}

In this section, we prove that every two-sided ideal $I$ in a \ca{} $A$ that admits a unital \stHom{} $M_2(\CC) \oplus M_3(\CC) \to A$ satisfies $I = A[A,I]A$, and we describe $[A,I]$ as the subspace spanned by certain square-zero elements in $I$;
see \autoref{prp:AAI-AI}.
If $I$ is a Dixmier ideal (in which case there is a well-defined ideal~$I^{\frac{1}{2}}$ satisfying $(I^{\frac{1}{2}})^2 = I$), then $I = [A,I] + [A,I^{\frac{1}{2}}]^2$;
see \autoref{prp:I-Dixmier}.
In this case every element in $I$ is a sum of square-zero elements in $I$ and of products of pairs of square-zero elements in~$I^{\frac{1}{2}}$.

\medskip

The next definition generalizes the notion of orthogonally factorizable square-zero elements in a ring (\cite[Deﬁnition~5.1]{GarThi23ZeroProdBalanced}) to a relative notion with respect to a two-sided ideal.

\begin{dfn}
\label{dfn:FN2}
Let $I$ be a two-sided ideal in a \ca{} $A$.
We let $N_2(I)$ denote the square-zero elements in $I$, that is, 
\[
\N_2(I) := \big\{ x \in I : x^2=0 \big\}.
\]
We further set
\[ 
\FN_2(A,I) 
:= \big\{ x \in I : \text{ there exist } a,b \in A, y \in I \text{ with } x=ayb \mbox{ and } ba=0 \big\}.
\] 
\end{dfn}

\begin{lma}
\label{prp:FN2PositiveSupport}
Let $I \subseteq A$ be a two-sided ideal in a \ca{}, and let $x \in \FN_2(A,I)$.
Then there exist $a,b \in A_+$ and $y \in \FN_2(A,I) \subseteq I$ such that
\[
x = ayb, \andSep ab=ba=0.
\]
\end{lma}
\begin{proof}
By definition, there are $c,d \in A$ and $z \in I$ such that
\[
x = czd, \andSep dc=0.
\]
Let $c=v|c|$ and $d=w|d|$ the polar decompositions in $A^{**}$ and set
\[
a := |c^*|^{\frac{1}{2}}, \quad
b := |d|^{\frac{1}{2}}, \andSep
y := |c^*|^{\frac{1}{2}}vzw|d|^{\frac{1}{2}}.
\]
By properties of the polar decomposition (see, for example, \cite[Proposition~2.1]{GarKitThi23arX:SemiprimeIdls}), we have $|c^*|^{\frac{1}{2}}v, w|d|^{\frac{1}{2}} \in A$, and thus $y \in A$.
Further, we have $x=ayb$.

Since $dc=0$, we have $(d^*d)^n(cc^*)^m=0$ for every $m,n \geq 1$.
It follows that $p(d^*d)q(cc^*)=0$ for every polynomials $p$ and $q$ with vanishing constant terms.
Using functional calculus, we get $f(d^*d)g(cc^*)=0$ for every continuous functions $f,g \colon \mathbb{R} \to \mathbb{R}$ with $f(0)=g(0)=0$.
In particular, we get $|d|^{\frac{1}{2}}|c^*|^{\frac{1}{2}}=0$.
Thus, we have $ba=0$, and then also $ab=0$.

To show that $y$ belongs to $\FN_2(A,I)$, consider the decomposition
\[
y = |c^*|^{\frac{1}{4}} \big( |c^*|^{\frac{1}{4}}vzw|d|^{\frac{1}{4}} \big) |d|^{\frac{1}{4}}.
\]
Then $|c^*|^{\frac{1}{4}}vzw|d|^{\frac{1}{4}}$ belongs to $I$, and we have also seen above that $|c^*|^{\frac{1}{4}}$ and $|d|^{\frac{1}{4}}$ have zero product, whence $y \in \FN_2(A,I)$.
\end{proof}

The next result is analogous to \cite[Lemma~5.2]{GarThi23ZeroProdBalanced}.

\begin{lma}
\label{prp:FN2}
Let $I \subseteq A$ be a two-sided ideal in a \ca{}.
Then
\[
\FN_2(A,I)
\subseteq [A,\FN_2(A,I)]
\subseteq [A,[A,I]]
\subseteq [A,I], \andSep
\FN_2(A,I)
\subseteq \N_2(I).
\]
\end{lma}
\begin{proof}
Let $x \in \FN_2(A,I)$.
By \autoref{prp:FN2PositiveSupport}, we can pick $a,b \in A_+$ and $y \in I$ such that $x=ayb$ and $ba=0$.
As in the proof of \autoref{prp:FN2PositiveSupport}, we see that $ba^{\frac{1}{2}}=0$.
It follows that $a^{\frac{1}{2}}yb \in \FN_2(A,I)$ and then
\[
x 
= ayb 
= \big[ a^{\frac{1}{2}}, a^{\frac{1}{2}}yb \big]
\in [A,\FN_2(A,I)].
\]

Since $\FN_2(A,I)$ is a subset of $I$, we have $[A,\FN_2(A,I)] \subseteq [A,I]$.
Combining these results, and using at the last step that $[A,I]$ is a Lie ideal, we get
\[
\FN_2(A,I)
\subseteq [A,\FN_2(A,I)]
\subseteq [A,[A,\FN_2(A,I)]]
\subseteq [A,[A,I]]
\subseteq [A,I].
\]

Given $x = ayb \in \FN_2(A,I)$, with $a,b \in A$ and $y \in I$ such that $ba=0$, we have
\[
x^2 = (ayb)(ayb) = 0, 
\]
and thus $x \in \N_2(I)$.
\end{proof}

\begin{rmk}
We note that \autoref{dfn:FN2} and \cite[Deﬁnition~5.1]{GarThi23ZeroProdBalanced} agree when considering the ideal $I=A$ in a \ca{} $A$.
Indeed, an element $x$ in $A$ belongs to $\FN_2(A)$ in the sense of \cite[Deﬁnition~5.1]{GarThi23ZeroProdBalanced} if there exist $s,t \in A$ such that $x = st$ and $ts=0$, while $x$ belongs to $\FN_2(A,A)$ as defined in \autoref{dfn:FN2} if there exist $a,y,b \in A$ such that $x=ayb$ and $ba=0$.

Given a factorization $x=ayb$ with $ba=0$, we can use $s=a$ and $t=yb$.
Conversely, given a factorization $x=st$ with $ts=0$, it follows that $|t|^{\frac{1}{2}}s=0$.
We consider the polar decomposition $t=v|t|$ in $A^{**}$.
Then $v|t|^{\frac{1}{2}}$ belongs to $A$, and we can use $a=s$, $y=v|t|^{\frac{1}{2}}$ and $b=|t|^{\frac{1}{2}}$.
\end{rmk}

Given a subset $V$ in a complex vector space, we use $\linSpan_\CC V$ to denote the linear subspace generated by $V$.
Following \cite[Deﬁnition~2.6]{GarThi23ZeroProdBalanced}, we say that a $\CC$-algebra~$A$ is \emph{zero-product balanced} if for all $a,b,c \in A$ we have
\[
ab \otimes c - a\otimes bc \in \linSpan_\CC \big\{ u \otimes v \in A \otimes_\CC A : uv=0 \big\} 
\subseteq  A \otimes_\CC A.
\]
This is closely related to the concept of `zero-product determination', for which we refer to \cite{Bre21BookZeroProdDetermined}.
By \cite[Corollary~2.12]{GarThi23ZeroProdBalanced}, a (not necessarily unital) \ca{} is zero-product balanced as a $\CC$-algebra if and only if it is zero-product determined.

We see that the class of \ca{s} considered in this section are zero-product balanced:

\begin{lma}
\label{prp:ZPBalanced}
Let $A$ be a unital \ca{} that admits a unital $*$-homomorphism $M_2(\CC) \oplus M_3(\CC) \to A$.
Then $A$ is zero-product balanced as a $\CC$-algebra.
\end{lma}
\begin{proof}
As in \autoref{pgr:Setup}, let $(e_{ij})_{i,j=1,2}$ and $(f_{ij})_{i,j=1,2,3}$ be the images in~$A$ of matrix units in $M_2(\CC)$ and $M_3(\CC)$.
Then $e_{11}+f_{11}$ and $e_{22}+f_{22}$ are two full orthogonal projections in~$A$.
This allows us to apply \cite[Corollary~3.8]{Rob16LieIdeals}, which shows that $A$ is generated by its projections as a $\CC$-algebra, and thus is zero-product determined (hence, zero-product balanced) by \cite[Theorem~2.3]{Bre21BookZeroProdDetermined}.
\end{proof}

A two-sided ideal $I \subseteq A$ in a \ca{} is said to be \emph{idempotent} if $I=I^2$, that is, if every element in $I$ is of the form $x_1y_1+\ldots+x_ny_n$ for some $n \in \NN$ and suitable $x_j,y_j \in I$.

\begin{prp}
\label{prp:SpanFN2}
Let $A$ be a \ca{} that is zero-product balanced as a $\CC$-algebra, and let $I \subseteq A$ be a two-sided ideal such that $I=AIA$ (for example, $A$ is unital or $I$ is idempotent).
Then
\[ 
[A,I] 
= [A,[A,I]]
= \linSpan_\CC \FN_2(A,I).
\] 
\end{prp}
\begin{proof}
The inclusions `$\supseteq$' hold by \autoref{prp:FN2}.
The proof of the converse inclusion is based on a combination of the methods used in the proofs of Proposition~2.14 and Theorem~5.3 in \cite{GarThi23ZeroProdBalanced}, which in turn are inspired by that of \cite[Theorem~9.1]{Bre21BookZeroProdDetermined}.
To verify that $[A,I] \subseteq \linSpan_\CC \FN_2(A,I)$, let $x \in A$ and $y \in I$.
Set $F := \linSpan_\CC \FN_2(A,I)$, and define
\[
\varphi_y \colon A \times A \to I/F, \quad
\varphi_y(a,b) := bya + F. \qquad\qquad (a,b \in A)
\]
Note that $\varphi_y$ is a well-defined $\CC$-bilinear map.
Further, $\varphi_y$ preserves zero-products since if $a,b \in A$ satisfy $ab=0$, then $bya \in \FN_2(A,I) \subseteq F$.
Using \cite[Proposition~2.8]{GarThi23ZeroProdBalanced} at the second step, we obtain that
\[
cyab + F
= \varphi_y(ab,c)
= \varphi_y(a,bc)
= bcya + F
\]
and thus
\[
bcya - cyab \in F = \linSpan_\CC \FN_2(A,I)
\]
for all $a,b,c \in A$.
The assumption that $I=AIA$ allows us to pick $n$ and $c_j,a_j \in A$ and $y_j \in I$ for $j=1,\ldots,n$ such that $y = \sum_j c_jy_ja_j$.
Then
\[
[x,y]
= \sum_{j=1}^n [x,c_jy_ja_j]
= \sum_{j=1}^n \big( xc_jy_ja_j - c_jy_ja_jx \big)
\in \linSpan_\CC \FN_2(A,I),
\]
which implies that $[A,I] \subseteq \linSpan_\CC \FN_2(A,I)$.
\end{proof}

\begin{qst}
\label{qst:AAI}
Does there exist a two-sided ideal $I$ in a \ca{} $A$ such that $[A,[A,I]] \neq [A,I]$?
\end{qst}

The next result verifies statement~(3) in \autoref{prp:MapFromM2M3}.

\begin{thm}
\label{prp:AAI-AI}
Let $A$ be a unital \ca{} that admits a unital \stHom{} $M_2(\CC) \oplus M_3(\CC) \to A$, and let $I \subseteq A$ be a (not necessarily closed) two-sided ideal.
Then
\[ 
I = A[A,I]A = A[A,I], \andSep 
[A,I]=[A,[A,I]] = \linSpan_\CC \FN_2(A,I).
\] 
\end{thm}
\begin{proof}
Considering $I$ as a Lie ideal in $A$, it follows from \autoref{prp:L-Related-IdlAL} that $A[A,I]A$ is the unique two-sided ideal in $A$ to which $I$ is related.
Since $I$ is clearly related to itself, we get $I=A[A,I]A$.
The inclusion $A[A,I] \subseteq A[A,I]A$ holds since $A$ is unital.
The converse inclusion holds since $[A,I]$ is a Lie ideal in $A$;
see, for example, \cite[Lemma~3.2]{GarLeeThi24arX:FullyNoncentral}.

Now the equality $[A,I]=[A,[A,I]]$ also follows from \autoref{prp:L-Related-IdlAL}.
Further, since~$A$ is zero-product balanced by \autoref{prp:ZPBalanced}, we have $[A,I]=\linSpan_\CC \FN_2(A,I)$ by \autoref{prp:SpanFN2}.
\end{proof}

\begin{rmk}
Let $A$ be a unital \ca{} that admits a unital \stHom{} $M_2(\CC) \oplus M_3(\CC) \to A$.
Applying \autoref{prp:AAI-AI} for the ideal $A$, we get
\[
A = A[A,A]A,
\]
which means that $A$ is generated by its commutators as an ideal.

In fact, $A$ is even generated by its commutators as a ring.
Indeed, by \cite[Theorem~5.15]{GarThi25GenByCommutators}, every element in $A$ is a sum of a commutator and the product of two commutators, and also the sum of three elements that are products of commutators.
\end{rmk}

The next result verifies statement~(4) in \autoref{prp:MapFromM2M3}.

\begin{thm}
\label{prp:CorrespondenceLie}
Let $A$ be a unital \ca{} that admits a unital \stHom{} $M_2(\CC) \oplus M_3(\CC) \to A$.
There is a natural bijection between the set $\calI$ of two-sided ideals in $A$ and the set 
\[
\calL := \big\{ L \subseteq A \text{ Lie ideal} : L = [A,L] \big\}
\]
given by the maps $I \mapsto [A,I]$ and $L \mapsto A[A,L]A$.
\end{thm}
\begin{proof}
We denote the maps by $\alpha \colon \calI \to \calL$ and $\beta \colon \calL \to \calI$.
Given $I \in \calI$,  using \autoref{prp:AAI-AI} at the second step, we have
\[
\alpha(I) = [A,I] = [A,[A,I]] = [A,\alpha(I)]
\]
which shows that $\alpha$ is well-defined.

Next, we show that $\alpha$ and $\beta$ are inverses of each other.
First, given $I \in \calI$, using \autoref{prp:AAI-AI} at the last step, we get
\[
\beta(\alpha(I))
= \beta([A,I])
= A[A,I]A
= I.
\]
Second, given $L \in \calL$, it follows from $L=[A,L]$ that $[A,L]=[A,[A,L]]$, and thus $L=[A,[A,L]]$.
Using this at the last step, and using \autoref{prp:L-Related-IdlAL} at the third step, we get
\[
\alpha(\beta(L))
= \alpha(A[A,L]A)
= [A,A[A,L]A]
= [A,[A,L]]
= L,
\]
as desired,
\end{proof}

Robert showed in \cite[Lemma~2.1]{Rob16LieIdeals} that every nilpotent element in a \ca{} is a sum of commutators.
With view towards \autoref{prp:AAI-AI}, we ask:

\begin{qst}
\label{qst:Nilpotent}
Let $I$ be a (not necessarily closed) two-sided ideal in a \ca{}~$A$.
Does every nilpotent element in $I$ belong to $[A,I]$?
\end{qst}

We give positive answers to \autoref{qst:Nilpotent} for arbitrary two-sided ideals in von Neumann algebras (\autoref{prp:NilpotentVNA}), and for semiprime two-sided ideals in \ca{s} (\autoref{prp:NilpotentSemiprime}).

\medskip

Following \cite[Definition~3.2]{GarKitThi23arX:SemiprimeIdls}, a \emph{Dixmier ideal} in a \ca{} $A$ is a two-sided ideal $I \subseteq A$ such that $I$ is positively spanned (that is, $I=\linSpan_\CC (I \cap A_+)$) and hereditary (that is, if $a,b \in A_+$ satisfy $a \leq b$ and $b \in I$, then $a \in I$) and strongly invariant (that is, if $x \in A$ satisfies $x^*x \in I$, then $xx^* \in I$).
Two-sided ideals in von Neumann algebras are automatically Dixmier ideals, which was essentially shown by Dixmier (see \cite[Proposition~3.4]{GarKitThi23arX:SemiprimeIdls}).

Given a Dixmier ideal~$I$ and $s \in (0,\infty)$, there is a unique Dixmier ideal whose set of positive elements is $\{ a^s : a \in I \cap A_+ \}$, and we use $I^s$ to denote this ideal;
see Proposition~3.7 and Definition~3.8 in \cite{GarKitThi23arX:SemiprimeIdls}.
We note that for $n \in \NN$ the ring-theoretic definition of $I^n$ as the two-sided ideal generated by $x_1\cdots x_n$ for $x_1,\ldots,x_n \in I$ agrees with the `new' definition of $I^n$ as the linear span of elements $a^n$ for $a \in I \cap A_+$.
This leads to a well-behaved theory of roots and powers for Dixmier ideals (\cite[Theorem~3.9]{GarKitThi23arX:SemiprimeIdls}), which we recall for the convenience of the reader.

\begin{thm}[Gardella, Kitamura, Thiel]
\label{prp:PowersDixmierIdl}
Let $I \subseteq A$ be a Dixmier ideal in a \ca.
Then
\[
(I^s)^t = I^{st} = (I^t)^s, \andSep
I^s I^t = I^{s+t} = I^t I^s
\]
for all $s,t \in (0,\infty)$.
\end{thm}

\begin{lma}
\label{prp:PowersPolarDecomp}
Let $I \subseteq A$ be a Dixmier ideal in a \ca, let $x \in I$ with polar decomposition $x = v|x|$ in $A^{**}$, and let $s \in (0,\infty)$.
Then
\[
v|x|^s, |x|^s, |x^*|^s, |x|^sv^* \in I^s.
\]
\end{lma}
\begin{proof}
Since $I$ is a Dixmier ideal, we have $x^* \in I$, and thus $x^*x, xx^* \in I^2$.
Then by \autoref{prp:PowersDixmierIdl}, we get
\[
|x| = (x^*x)^{\frac{1}{2}} \in (I^2)^{\frac{1}{2}}=I, \andSep
|x^*| = (xx^*)^{\frac{1}{2}} \in (I^2)^{\frac{1}{2}}=I.
\]
It follows that $|x|^s, |x^*|^s \in I^s$.

By \cite[Proposition~3.10]{GarKitThi23arX:SemiprimeIdls}, given a Dixmier ideal $J \subseteq A$, an element $y \in A$ belongs to $J^{\frac{1}{2}}$ if and only if $y^*y \in J$.
Applying this for the Dixmier ideal $I^{2s}$, and using that
\[
\big( v|x|^s \big)^*\big( v|x|^s \big) = |x|^{2s} \in I^{2s},
\]
we deduce that $v|x|^s \in (I^{2s})^{\frac{1}{2}} = I^s$.
Since $I^s$ is a Dixmier ideal, we further have $|x|^sv^* = (v|x|^s)^* \in I^s$.
\end{proof}

The next result showcases a situation when every element in a Dixmier ideal $I$ is a sum of square-zero element in $I$ and products of pairs of square-zero elements in~$I^{\frac{1}{2}}$.
In \autoref{prp:IdlVNA}, we strengthen this result for ideals in von Neumann algebras.

\begin{thm}
\label{prp:I-Dixmier}
Let $A$ be a unital \ca{} that admits a unital \stHom{} $M_2(\CC) \oplus M_3(\CC) \to A$, and let $I \subseteq A$ be a Dixmier ideal.
Then
\[
I = [A,I] + [A,I^{\frac{1}{2}}]^2.
\]
\end{thm}
\begin{proof}
We first show that
\[
A \cdot \FN_2(A,I) 
\subseteq [A,I] + [A,I^{\frac{1}{2}}]^2.
\]
Let $c \in A$ and $x \in \FN_2(A,I)$.
By \autoref{prp:FN2PositiveSupport}, we can pick $a,b \in A_+$ and $y \in I$ such that $x=ayb$ and $ab=ba=0$.
Set $z := yb \in I$, and let $z=v|z|$ be the polar decomposition in $A^{**}$.
Then
\[
cx
= \big[ ca^{\frac{1}{2}},a^{\frac{1}{2}}z \big] + a^{\frac{1}{2}}zca^{\frac{1}{2}}
= \big[ ca^{\frac{1}{2}},a^{\frac{1}{2}}z \big] + \big( a^{\frac{1}{2}}v|z|^{\frac{1}{2}} \big)\big( |z|^{\frac{1}{2}}ca^{\frac{1}{2}} \big).
\]

By \autoref{prp:PowersPolarDecomp}, $v|z|^{\frac{1}{2}}$ and $|z|^{\frac{1}{2}}$ belong to~$I^{\frac{1}{2}}$.
Further, since $za=0$, we get $|z|^{\frac{1}{2}}a^{\frac{1}{2}}=0$, and hence
\[
a^{\frac{1}{2}}v|z|^{\frac{1}{2}}
= \big[ a^{\frac{1}{2}},v|z|^{\frac{1}{2}} \big]
\in [A,I^{\frac{1}{2}}], \andSep
|z|^{\frac{1}{2}}ca^{\frac{1}{2}}
= \big[ -ca^{\frac{1}{2}}, |z|^{\frac{1}{2}} \big]
\in [A,I^{\frac{1}{2}}].
\]

It follows that $cx$ belongs to $[A,I] + [A,I^{\frac{1}{2}}]^2$.

\medskip

Using \autoref{prp:AAI-AI} at the first and second step, and using the above at the last step, we get
\[
I 
= A[A,I]
= \linSpan_\CC \Big( A\cdot\FN_2(A,I) \Big)
\subseteq [A,I] + [A,I^{\frac{1}{2}}]^2.
\]

On the other hand, we clearly have $[A,I] \subseteq I$.
Using \autoref{prp:PowersDixmierIdl}, we also have $[A,I^{\frac{1}{2}}]^2 \subseteq (I^{\frac{1}{2}})^2 = I$.
\end{proof}

\section{Square-zero elements in semiprime ideals in C*-algebras}

In this section, we study the subspace $N$ generated by the square-zero elements in a semiprime ideal in a \ca{} $A$.
We show that $ANA$, the two-sided ideal generated by $N$, satisfies $ANA = [A,N] + N^2$;
see \autoref{prp:N2-semiprime}.
This generalize results from \cite[Section~4]{GarThi24PrimeIdealsCAlg}.

We then characterize semiprimeness for two-sided ideals in a \ca{} whose unit is weakly divisible of degree~$2$;
see \autoref{prp:CharSemiprime}.
In particular, such an ideal is semiprime if and only if it is generated by its commutators as an ideal (or as a ring).

\medskip

We begin with some preparatory results about commutators involving Dixmier ideals, which might be of use in future work.

\begin{lma}
\label{prp:AI-I12-I12}
Let $I \subseteq A$ be a Dixmier ideal in a \ca.
Then
\[
[A,I]
\subseteq \big[ I^{\frac{1}{2}},I^{\frac{1}{2}} \big]
\subseteq I.
\]
\end{lma}
\begin{proof}
By \autoref{prp:PowersDixmierIdl}, we have $I = I^{\frac{1}{2}}I^{\frac{1}{2}}$.
Using that $[a,xy]=[ax,y]+[ya,x]$ for $a \in A$ and $x,y \in I$, it follows that
\[
[A,I]
= \big[ A,I^{\frac{1}{2}}I^{\frac{1}{2}} \big]
\subseteq \big[ AI^{\frac{1}{2}},I^{\frac{1}{2}} \big] + \big[ I^{\frac{1}{2}}A,I^{\frac{1}{2}} \big]
\subseteq \big[ I^{\frac{1}{2}},I^{\frac{1}{2}} \big].
\]
Further, we have $[I^{\frac{1}{2}},I^{\frac{1}{2}}] \subseteq I^{\frac{1}{2}}I^{\frac{1}{2}} = I$.
\end{proof}

Recall that $\N_2(I)$ denotes the set of square-zero elements in an ideal $I$.

\begin{lma}
\label{prp:N2-Dixmier}
Let $I \subseteq A$ be a Dixmier ideal in a \ca. 
Then
\[
\N_2(I) \subseteq \linSpan_\CC \big[ \N_2(I^{\frac{1}{2}}),\N_2(I^{\frac{1}{2}}) \big].
\]
\end{lma}
\begin{proof}
Let $x \in \N_2(I)$, and let $x=v|x|$ be the polar decomposition in $A^{**}$.
We have $x = |x^*|v = |x^*|^{\frac{1}{2}}v|x|^{\frac{1}{2}}$.
Since $x^2=0$, the elements $|x|$ and $|x^*|$ are orthogonal, and hence so are $|x|^{\frac{1}{2}}$ and $|x^*|^{\frac{1}{2}}$. 
Therefore
\[
x = \big[ \tfrac{1}{2}v|x|^{\frac{1}{2}},|x|^{\frac{1}{2}} - |x^*|^{\frac{1}{2}} \big].
\]

By \autoref{prp:PowersPolarDecomp}, the element $\tfrac{1}{2}v|x|^{\frac{1}{2}}$ belongs to~$I^{\frac{1}{2}}$, and thus $\tfrac{1}{2}v|x|^{\frac{1}{2}} \in \N_2(I^{\frac{1}{2}})$.
Set 
\[
y := |x|^{\frac{1}{2}} - |x^*|^{\frac{1}{2}} + v|x|^{\frac{1}{2}} - v^*|x^*|^{\frac{1}{2}}.
\]
It is elementary (but tedious) to check that $y^2=0$.
Note that $y$ belongs to $I^{\frac{1}{2}}$ by \autoref{prp:PowersPolarDecomp}, and thus $y \in \N_2(I^{\frac{1}{2}})$.
Using that $v|x|^{\frac{1}{2}}$ and $v^*|x^*|^{\frac{1}{2}}$ also belong to $N_2(I^{\frac{1}{2}})$, it follows that
\[
|x|^{\frac{1}{2}} - |x^*|^{\frac{1}{2}}
= y - v|x|^{\frac{1}{2}} + v^*|x^*|^{\frac{1}{2}}
\in \linSpan_\CC N_2(I^{\frac{1}{2}}),
\]
and thus
\[
x 
= \big[ \tfrac{1}{2}v|x|^{\frac{1}{2}},|x|^{\frac{1}{2}} - |x^*|^{\frac{1}{2}} \big]
\in \linSpan_\CC \big[ N_2(I^{\frac{1}{2}}),N_2(I^{\frac{1}{2}}) \big],
\]
as desired.
\end{proof}

A two-sided ideal $I$ in a ring $R$ is said to be \emph{semiprime} if for all two-sided ideals $J \subseteq R$ we have $J \subseteq I$ whenever $J^2 \subseteq I$.
We refer to \cite[Section~10]{Lam01FirstCourse2ed} for details.
In \cite[Section~5]{GarKitThi23arX:SemiprimeIdls}, together with Gardella and Kitamura, we showed that a two-sided ideal $I$ in a \ca{} is semiprime if and only if it is idempotent (that is, $I=I^2$), if and only if it is a Dixmier ideal and satisfies $I=I^{\frac{1}{2}}$.

The next result is a generalization of \cite[Theorem~4.2]{GarThi24PrimeIdealsCAlg} from the study of square-zero elements in \ca{s} to square-zero elements in semiprime ideals in \ca{s}.
Recall that $\widetilde{A}$ denotes the minimal unitization of a \ca{}~$A$.

\begin{thm}
\label{prp:N2-semiprime}
Let $I \subseteq A$ be a semiprime ideal in a \ca, let $N$ denote the subspace generated by $\N_2(I)$, and let $J := \widetilde{A}N\widetilde{A}$ denote the two-sided ideal of $A$ generated by $N$.
Then $J$ is semiprime, and we have
\[
N  = [N,N] \subseteq N^2, \quad
J = [J,J]^2 = AN = [A,N]^2 = [A,N] + N^2,
\]
and
\[
[J,J] = [A,J] = [A,[A,N]] = [[J,J],[J,J]] = [[A,N],[A,N]]. 
\]
\end{thm}
\begin{proof}
We will use that $I$ is a Dixmier ideal satisfying $I=I^{\frac{1}{2}}$.
Applying \autoref{prp:N2-Dixmier} at the second step, we have
\[
N 
= \linSpan_\CC \N_2(I) 
\subseteq \linSpan_\CC \big[ \N_2(I^{\frac{1}{2}}),\N_2(I^{\frac{1}{2}}) \big]
= \linSpan_\CC \big[ \N_2(I),\N_2(I) \big]
= [N,N].
\]
The inclusion $[N,N] \subseteq N^2$ follows since every commutator in $[N,N]$ belongs to~$N^2$.
To verify that $[N,N] \subseteq N$, let $x,y \in \N_2(I)$.
Then
\[
[x,y] = (1+x)y(1-x) - y + xyx.
\]
Each of the three summands is a square-zero element in $I$, which shows $[x,y] \in N$.

\medskip

We have
\[
J := \widetilde{A}N\widetilde{A} 
\subseteq \widetilde{A}N^2\widetilde{A}
\subseteq \big( \widetilde{A}N \big)\big( N\widetilde{A} \big)
\subseteq J^2,
\]
which shows that $J$ is idempotent and therefore semiprime by \cite[Theorem~A]{GarKitThi23arX:SemiprimeIdls}.

The inclusion $AN \subseteq \widetilde{A}N\widetilde{A} = J$ is clear.
Conversely, given $a \in \widetilde{A}$, $x \in \N_2(I)$ and a unitary $u \in \widetilde{A}$, using that $u^*xu \in \N_2(I)$ and that $N \subseteq N^2$, we have
\[
axu = au(u^*xu)
\in \widetilde{A} \N_2(I)
\subseteq \widetilde{A} N
\subseteq \widetilde{A} N^2
\subseteq AN.
\]
Since every element in $\widetilde{A}$ is a linear combination of four unitaries, we get $J=AN$.

The inclusion $[A,N] + N^2 \subseteq \widetilde{A}N\widetilde{A} = J$ is clear.
For the converse inclusion, we use a similar argument as in the proof of \autoref{prp:I-Dixmier} to show that $AN \subseteq [A,N] + N^2$.
Given $a \in A$ and $x \in \N_2(I)$, let $x=v|x|$ be the polar decomposition in $A^{**}$.
Since $I$ is semiprime, it is Dixmier and $I = I^{\frac{1}{2}} = I^{\frac{1}{4}}$.
By \autoref{prp:PowersPolarDecomp}, we have $v|x|^{\frac{1}{2}} \in I^{\frac{1}{2}}$ and thus $v|x|^{\frac{1}{2}} \in \N_2(I)$.
Similarly, we have $v|x|^{\frac{1}{4}}, |x|^{\frac{1}{4}} a|x^*|^{\frac{1}{2}} \in \N_2(I)$.
Using that $x=|x^*|^{\frac{1}{2}}v|x|^{\frac{1}{2}}$, we get
\[
ax
= \big[ a|x^*|^{\frac{1}{2}},v|x|^{\frac{1}{2}} \big] 
+ \big( v|x|^{\frac{1}{4}} \big) \big(|x|^{\frac{1}{4}} a|x^*|^{\frac{1}{2}} \big)
\in [A,\N_2(I)] + \N_2(I)^2.
\]
Using that $J=AN$, we get $J = [A,N] + N^2$.

The inclusion $[A,N]^2 \subseteq \widetilde{A}N\widetilde{A} = J$ is clear.
Conversely, using that $N=[N,N]$, we have
\[
[A,N]
= [A,[N,N]]
\subseteq [N,[A,N]]
\subseteq [[A,N],[A,N]]
\subseteq [A,N]^2.
\]
We also have $N^2 = [N,N]^2  \subseteq [A,N]^2$ and thus
\[
J = [A,N] + N^2
\subseteq [A,N]^2.
\]

\medskip

The inclusion $[J,J] \subseteq [A,J]$ is clear.
Using that $J=J^2$ and $[a,xy]=[ax,y]+[ya,x]$ for $a \in A$ and $x,y \in J$, we have
\[
[A,J]
= [A,J^2]
\subseteq [AJ,J] + [JA,J]
\subseteq [J,J].
\]
This shows that $[J,J] =  [A,J]$.
We get
\[
[[A,N],[A,N]] 
\subseteq [A,[A,N]]
\subseteq [A,J]
= [J,J].
\]
We have
\[
\big[ [A,N],N^2 \big]
\subseteq [A,N^2]
\subseteq [AN,N] + [NA,N]
\subseteq [A,N].
\]
and similarly $[N^2,N^2] \subseteq [A,N^2] \subseteq [A,N]$.
We also have
\[
[A,N]
= [A,[N,N]]
\subseteq [N,[A,N]]
= [[N,N],[A,N]]
\subseteq [[A,N],[A,N]].
\]
Using that $J = [A,N] + N^2$, it follows that
\begin{align*}
[J,J] 
&\subseteq 
\big[ [A,N],[A,N] \big]
+ \big[ N^2,[A,N] \big]
+ \big[ [A,N],N^2 \big]
+ \big[ N^2,N^2 \big] \\
&\subseteq \big[ [A,N],[A,N] \big].
\end{align*}
Using that $[A,N] \subseteq [A,J]=[J,J]$, we also have
\[
[J,J]
= \big[ [A,N],[A,N] \big]
\subseteq \big[ [J,J],[J,J] \big].
\]
Since the converse inclusion is clear, we get
\[
[J,J] = [A,J] = [A,[A,N]] = [[J,J],[J,J]] = [[A,N],[A,N]].
\]

Finally, since $[A,N] \subseteq [J,J]$ and $J = [A,N]^2$, we have $J \subseteq [J,J]^2$ and thus $J = [J,J]^2$.
\end{proof}

\begin{lma}
\label{prp:DixmierNilpotent}
Let $I \subseteq A$ be a Dixmier ideal in a \ca, and let $x \in I$ be nilpotent.
Then $x \in [I^{\frac{1}{2}},I^{\frac{1}{2}}]$.
\end{lma}
\begin{proof}
As in the proof of \cite[Lemma~2.1]{Rob16LieIdeals}, we proceed by induction on the degree of nilpotency.
To start, let $x \in I$ be a square-zero element, and let $x=v|x|$ be the polar decomposition of $x$ in $A^{**}$.
By \autoref{prp:PowersPolarDecomp}, the elements $v|x|^{\frac{1}{2}}$ and $|x|^{\frac{1}{2}}$ belong to $I^{\frac{1}{2}}$, and thus
\[
x
= \big[ v|x|^{\frac{1}{2}}, x^{\frac{1}{2}} \big]
\in \big[ I^{\frac{1}{2}},I^{\frac{1}{2}} \big].
\]

Next, assume that for some $k \geq 2$ we have shown that every element $x \in I$ with $x^k=0$ belongs to $[I^{\frac{1}{2}},I^{\frac{1}{2}}]$.
Let $x \in I$ satisfy $x^{k+1}=0$, and let $x=v|x|$ be the polar decomposition of $x$ in $A^{**}$.
Set $y:=|x|^{\frac{1}{2}}v|x|^{\frac{1}{2}}$, which belongs to $A$.
As in the proof of \cite[Lemma~2.1]{Rob16LieIdeals}, one shows that $y^k=0$.
By \autoref{prp:PowersPolarDecomp}, we have $v|x|^{\frac{1}{2}}, |x|^{\frac{1}{2}} \in I^{\frac{1}{2}}$.
By assumption of the induction, we have $y \in [I^{\frac{1}{2}},I^{\frac{1}{2}}]$, and thus
\[
x 
= \big[ v|x|^{\frac{1}{2}},|x|^{\frac{1}{2}} \big] + y 
\in \big[ I^{\frac{1}{2}},I^{\frac{1}{2}} \big]
\]
as desired.
\end{proof}

\begin{prp}
\label{prp:NilpotentSemiprime}
Let $I \subseteq A$ be a semiprime ideal in a \ca, and let $x \in I$ be nilpotent.
Then $x \in [A,I]$.
\end{prp}
\begin{proof}
Since $I$ is semiprime, it is a Dixmier ideal and satisfies $I = I^{\frac{1}{2}}$.
Applying \autoref{prp:DixmierNilpotent}, we have
\[
x 
\in [I^{\frac{1}{2}},I^{\frac{1}{2}}]
\subseteq [A,I^{\frac{1}{2}}]
= [A,I],
\]
as desired.
\end{proof}

\begin{thm}
\label{prp:CharSemiprime}
Let $A$ be a unital \ca{} that admits a unital \stHom{} $M_2(\CC) \oplus M_3(\CC) \to A$, and let $I \subseteq A$ be a two-sided ideal.
The following statements are equivalent:
\begin{enumerate}
\item
The ideal $I \subseteq A$ is semiprime.
\item
We have $I=A[I,I]A$, that is, $I$ is generated by its commutators as an ideal in $A$.
\item
We have $I=\widetilde{I}[I,I]\widetilde{I}$, that is, $I$ is generated by its commutators as an ideal.
\item
We have $I=[I,I]^2$.
\item
We have $[I,I]=[[I,I],[I,I]]$.
\end{enumerate}
\end{thm}
\begin{proof}
To show that~(2) implies~(1), assume that $I$ is generated by $[I,I]$ as an ideal in $A$.
We clearly have $[I,I] \subseteq I^2$, and therefore
\[
I = A[I,I]A \subseteq AI^2A = I^2.
\]
It follows that $I$ is idempotent, and thus semiprime by \cite[Theorem~A]{GarKitThi23arX:SemiprimeIdls}.

It is clear that~(4) implies~(3), which in turn implies~(2).

Assume that $I$ is semiprime, and let $N$ denote the subspace generated by the square-zero elements in~$I$.
Applying \autoref{prp:AAI-AI}, we have 
\[
I 
= A[A,I]A
= \linSpan_\CC A\FN_2(A,I)A
\subseteq ANA.
\]
Now it follows from \autoref{prp:N2-semiprime} that
\[
I 
= ANA 
= [ANA,ANA]^2
= [I,I]^2,
\]
which verifies~(4), and that
\[
[I,I]
= [ANA,ANA]
= [[ANA,ANA],[ANA,ANA]]
= [[I,I],[I,I]], \andSep
\]
which verifies~(5).

Finally, assuming~(5) let us verify~(1).
Using that $[a,xy]=[ax,y]+[ya,x]$ for $a \in A$ and $x,y \in I$, we have
\[
[A,I^2] \subseteq [AI,I] + [IA,I] \subseteq [I,I].
\]
Applying \autoref{prp:AAI-AI} for the ideal $I^2$, we get
\[
I^2
= A[A,I^2]A
\subseteq A[I,I]A
\subseteq A[[I,I],[I,I]]A
\subseteq I^4.
\]
It follows that $I^2=I^4$, and thus semiprime by \cite[Theorem~A]{GarKitThi23arX:SemiprimeIdls}.
\end{proof}

\section{Lie ideals in properly infinite C*-algebras}

Let $A$ be a \ca{} with a unit that is weakly divisible of degree~$2$, that is, that admits a unital \stHom{} $M_2(\CC) \oplus M_3(\CC) \to A$, and let $L \subseteq A$ be a (not necessarily closed) Lie ideal.
We have seen in \autoref{prp:L-Related-IdlAL} that $L$ is related to the two-sided ideal $I := A[A,L]A$.
Further, by \autoref{prp:CharLieEmbr}, $L$ is commutator equivalent to $I$ if and only if $[A,L] = [A,[A,L]]$.

This raises the question when Lie ideals $L$ in $A$ satisfy $[A,L] = [A,[A,L]]$.
A sufficient condition is that $A = Z(A) + [A,A]$, where $Z(A)$ denotes the center of~$A$;
see \autoref{prp:ZA-AA}.
We observe in \autoref{prp:Sufficient-ZA-AA} that this condition is satisfies in a number of situations, in particular for all properly infinite \ca{s}.

We then prove the main result of the paper:
Every Lie ideal in a unital, properly infinite \ca{} is commutator equivalent to a (unique) two-sided ideal;
see \autoref{prp:PropInf}.

\begin{lma}
\label{prp:ZA-AA}
Let $A$ be a unital \ca{} such that $A = Z(A) + [A,A]$, and let $L \subseteq A$ be a Lie ideal.
Then
\[
[A,L] = [A,[A,L]].
\]
\end{lma}
\begin{proof}
The inclusion `$\supseteq$' is clear.
Conversely, using the assumption at the first step, and using the Jacobi identity at the last step, we get
\[
[A,L] 
= \big[ Z(A)+[A,A],L \big]
= \big[ [A,A],L \big]
\subseteq \big[ A,[A,L] \big],
\]
as desired.
\end{proof}

Following Winter \cite{Win12NuclDimZstable}, a \ca{} $A$ is said to be \emph{pure} if its Cuntz semigroup $\Cu(A)$ is almost unperforated and almost divisible.
Pure \ca{s} form a robust class \cite{AntPerThiVil24arX:PureCAlgs} that includes every \ca{} that tensorially absorbs the Jiang-Su algebra $\mathcal{Z}$, that is, $A \cong A \otimes \mathcal{Z}$.
By \cite[Theorem~7.3.11]{AntPerThi18TensorProdCu}, a \ca{}~$A$ is pure if and only if $\Cu(A) \cong \Cu(A) \otimes \Cu(\mathcal{Z})$.

By a \emph{quasitrace} of a \ca{} $A$ we mean a lower-semicontinuous, $[0,\infty]$-valued $2$-quasitrace.
By Haagerup's theorem \cite{Haa14Quasitraces}, every bounded quasitrace on a unital, exact \ca{} is a trace, and this was later extended by Kirchberg \cite{Kir97TracesProjectionlessSimple}, who showed that arbitrary quasitraces on exact \ca{s} are traces.

A unital \ca{} $A$ is \emph{weakly central} if the center separates maximal ideals:
Maximal ideals $M_1,M_2 \subseteq A$ satisfy $M_1=M_2$ whenever $M_1 \cap Z(A) = M_2 \cap Z(A)$.
By Vesterstr{\o}m's theorem, $A$ is weakly central if and only if $A$ has the \emph{centre-quotient property}: For every closed, two-sided ideal $I \subseteq A$, the quotient map $A \to A/I$ maps $Z(A)$ onto $Z(A/I)$.

A unital \ca{} $A$ has the \emph{Dixmier property} if for every element $a \in A$ the closed convex hull of the unitary orbit $\{uau^* : u \in A \text{ unitary} \}$ has nonempty intersection with the center $Z(A)$.
It was shown in \cite[Theorem~1.1]{ArcRobTik17Dixmier} that this property holds if and only if $A$ is weakly central, every simple quotient of $A$ has at most one tracial state, and every extreme tracial state of $A$ factors through a simple quotient.
Dixmier showed that every von Neumann algebra has the property named after him.

\begin{prp}
\label{prp:Sufficient-ZA-AA}
Let $A$ be a unital \ca.
Then $A = Z(A) + [A,A]$ in each of the following cases:
\begin{enumerate}
\item
If $A$ is pure, every quasitrace on $A$ is a trace, and $A$ has the Dixmier property.
\item
If $A$ has no tracial states (in particular, if $A$ is properly infinite).
\item
If $A$ is a von Neumann algebra.
\end{enumerate}
\end{prp}
\begin{proof}
(1)
If $A$ has the Dixmier property, then every element in $A$ is a Dixmier element in the sense of \cite{ArcGogRob23LocalVarDixmier}, and thus it follows from \cite[Lemma~4.10]{ArcGogRob23LocalVarDixmier} that $A=Z(A)+\overline{[A,A]}$.

If $A$ is pure and every quasitrace on $A$ is a trace, then $[A,A]$ is closed by \cite[Theorem~1.1]{NgRob16CommutatorsPureCa}.
Combining both results, we get $A = Z(A) + [A,A]$.

(2)
By \cite[Theorem~1]{Pop02FiniteSumsCommutators}, we have $A=[A,A]$ if (and only if) $A$ has no tracial states.

(3)
It was shown in the proof of \cite[Theorem~5.19]{BreKisShu08LieIdeals} that every von Neumann algebra $A$ satisfies $A=Z(A)+[A,A]$.
\end{proof}

\begin{thm}
\label{prp:PropInf}
Let $A$ be a unital, properly infinite \ca{}, and let $L \subseteq A$ be a Lie ideal.
Then for the two-sided ideal $I:=A[A,L]A$, we have
\[
[A,L]=[A,[A,L]]=[A,I].
\]
In particular, $L$ is commutator equivalent to $I$, and hence also embraced by~$I$, and related to~$I$.

Further, $I$ is the only two-sided ideal of $A$ to which $L$ is related.
\end{thm}
\begin{proof}
Since $A$ is unital and properly infinite, there exists a unital \stHom{} from the Cuntz algebra $\mathcal{O}_\infty$ to~$A$; 
see \cite[Proposition~III.1.3.3]{Bla06OpAlgs}.
Further, it is well-known that $\mathcal{O}_\infty$ admits a unital \stHom{} from $M_2(\CC) \oplus M_3(\CC)$, using for example \autoref{prp:RR0} and that $\mathcal{O}_\infty$ is a unital \ca{} of real rank zero that admits no characters.
It follows that $A$ admits a unital \stHom{} $M_2(\CC) \oplus M_3(\CC) \to A$.

We may therefore apply \autoref{prp:L-Related-IdlAL} and obtain that $L$ is related to $I$, and that $I$ is the only two-sided ideal to which $L$ is related.
Further, since $A$ is unital and properly infinite, we have $A=[A,A]$ by Pop's theorem \cite{Pop02FiniteSumsCommutators}, and thus $[A,L]=[A,[A,L]]$ by \autoref{prp:ZA-AA}.
It therefore follows from \autoref{prp:CharLieEmbr} that $L$ is commutator equivalent to $I$ (and thus also embraced by $I$).
\end{proof}

\begin{cor}
\label{prp:PropInfDecompLie}
Let $A$ be a unital, properly infinite \ca{}.
Then an additive subgroup $L \subseteq A$ is a Lie ideal if and only if $[A,I] \subseteq L \subseteq T([A,I])$ for some (unique) two-sided ideal $I \subseteq A$.
We get a decomposition of the family $\mathcal{L}$ of all Lie ideals in~$A$ as:
\[
\mathcal{L} = \bigsqcup_{I \in \calI} \big\{ L \subseteq A \text{ additive subgroup} : [A,I] \subseteq L \subseteq T([A,I]) \big\}
\]
where $\calI$ denotes the lattice of two-sided ideals in $A$.
\end{cor}

\begin{rmk}
The proof of \autoref{prp:PropInf} can easily be adapted to show the following:
In a \ca{} $A$ that admits a unital \stHom{} $M_2(\CC) \oplus M_3(\CC) \to A$ and such that $A=Z(A)+[A,A]$, every Lie ideal $L$ is commutator equivalent to the two-sided ideal $A[A,L]A$.

We note that a unital \ca{} admits no characters whenever it is pure or has no tracial stetes.
Therefore, with view towards \autoref{qst:NoChar} and \autoref{prp:Sufficient-ZA-AA}, we expect that every Lie ideal $L$ is commutator equivalent to $A[A,L]A$ whenever~$A$ is a unital, pure,  \ca{} with the Dixmier property and such that every quasitrace on $A$ is a trace, or whenever~$A$ is unital and has no tracial states.
\end{rmk}

\begin{qst}
Given Lie ideals $L,M \subseteq A$ in a unital, properly infinite \ca, what is the relationship between the two-sided ideals generated by $[L,M]$, by $[A,[L,M]]$, and by $[A,L]$ and $[A,M]$?
\end{qst}

\section{Lie ideals and two-sided ideals in von Neumann algebras}

In this section, we show that every (not necessarily closed) Lie ideal $L$ in a von Neumann algebra $M$ is commutator equivalent to the two-sided ideal $I := M[M,L]M$, that is, $[M,L]=[M,I]$;
see \autoref{prp:VNA}.
This recovers \cite[Theorem~5.19]{BreKisShu08LieIdeals}.
Moreover, if $M$ has zero commutative summand, then we show that $I$ is the only two-sided ideal to which $L$ is related, which solves \cite[Problem~5.21]{BreKisShu08LieIdeals}.

Given a two-sided ideal $I$ in a von Neumann algebra with zero commutative summand, every element in $I$ is a sum of products of pairs of square-zero elements in $I^{\frac{1}{2}}$;
see \autoref{prp:IdlVNA}.

\begin{thm}
\label{prp:VNA}
Let $M$ be a von Neumann algebra, and let $L \subseteq M$ be a Lie ideal.
Then, for the two-sided ideal $I:=M[M,L]M$, we have
\[
[M,L]
= [M,[M,L]]
= [M,I].
\]
In particular, $L$ is commutator equivalent to $I$, and hence also embraced by~$I$, and related to~$I$.

If $M$ has zero commutative summand, then $I$ is the only two-sided ideal of $M$ to which $L$ is related.
\end{thm}
\begin{proof}
Let $M=M_0 \oplus M_1$ be the (unique) decomposition of $M$ into a commutative summand $M_0$ and a summand $M_1$ that admits no characters.
By \autoref{prp:RR0}, there is a unital \stHom{} $M_2(\CC) \oplus M_3(\CC) \to M_1$.
Set
\[
L_1 := M_1 \cap L, \andSep
I_1 := M_1 \cap I.
\]
Then $L_1$ is a Lie ideal in $M_1$, and we have
\[
[M,L] = [M_1,L_1], \andSep
I = I_1 = M_1[M_1,L_1]M_1.
\]

We may apply \autoref{prp:L-Related-IdlAL} for the Lie ideal $L_1$ in $M_1$ and obtain that $L_1$ is related to $I_1$, and that $I_1$ is the only two-sided ideal of $M_1$ to which $L_1$ is related.
It follows that $L$ is related to $I$, and if $M_0$ is zero, then $I$ is also unique with this property.
(If $M_0$ is nonzero, then $M_0 \oplus I$ is another two-sided ideal of $M$ to which~$L$ is related.)

As shown in the proof of \cite[Theorem~5.19]{BreKisShu08LieIdeals}, we have $M=Z(M)+[M,M]$.
It follows that $[M,L] = [M,[M,L]]$ by \autoref{prp:ZA-AA}, and then
\[
[M,L]
= [M,[M,L]]
= [M,I].
\]
by \autoref{prp:CharLieEmbr}.
\end{proof}

Recall that $\N_2(I)$ denotes the square-zero elements, and we use $\Nil(I)$ to denote the set of nilpotent elements in an ideal $I$.
Recall the definition of $\FN_2(M,I)$ from \autoref{dfn:FN2}.

\begin{prp}
\label{prp:NilpotentVNA}
Let $I \subseteq M$ be a two-sided ideal in a von Neumann algebra.
Then
\[
[M,I]
= \linSpan_\CC \N_2(I)
= \linSpan_\CC \Nil(I)
= \linSpan_\CC \FN_2(M,I).
\]
\end{prp}
\begin{proof}
Let $M=M_0 \oplus M_1$ denote the (unique) decomposition of $M$ into a commutative summand $M_0$ and a summand $M_1$ that admits no characters.
Then
\begin{align*}
[M,I] &= [M_1,M_1 \cap I], \quad
\N_2(I) = \N_2(M_1 \cap I), \\
\Nil(I) &= \Nil(M_1 \cap I), \andSep
\FN_2(M,I) = \FN_2(M_1,M_1\cap I),
\end{align*}
which allows us to reduce the problem to studying the ideal $M_1 \cap I$ in $M_1$.

Without loss of generality, we may therefore assume that $M$ has zero commutative summand.
Then $M$ admits a unital \stHom{} $M_2(\CC) \oplus M_3(\CC) \to M$ by \autoref{prp:RR0}.
Using \autoref{prp:AAI-AI} at the first step, we get
\[
[M,I]
= \linSpan_\CC \FN_2(M,I)
\subseteq \linSpan_\CC \N_2(I)
\subseteq \linSpan_\CC \Nil(I).
\]

It remains to show that $\Nil(I) \subseteq [M,I]$.
As in the proof of \autoref{prp:DixmierNilpotent} (and \cite[Lemma~2.1]{Rob16LieIdeals}), we proceed by induction on the degree of nilpotency.
To start, let $x \in I$ be a square-zero element.
Let $x=v|x|$ be the polar decomposition in $M$.
Then $|x|=v^*x \in I$.
Since $x^2=0$, we have $|x|v=0$, and thus
\[
x 
= \big[ v,|x| \big]
\in [M,I].
\]

Next, assume that for some $k \geq 2$ we have shown that every element $x \in I$ with $x^k=0$ belongs to $[M,I]$.
Let $x \in I$ satisfy $x^{k+1}=0$.
Again, let $x=v|x|$ be the polar decomposition in $M$.
We will show that the element $y:=|x|v$ satisfies $y^k=0$.
Since $x^{k+1}=0$ and $|x|=v^*v|x|$, we have
\[
y^k|x|
= |x|v \cdots |x|v|x|
= v^*v|x|v \cdots |x|v|x|
= v^*x^{k+1}
= 0.
\]
For every polynomial $p$ with vanishing constant term it follows that $y^kp(|x|)=0$.
Using functional calculus, we get $y^k|x|^{\frac{1}{n}}$ for every $n\geq 1$.
The sequence $(|x|^{\frac{1}{n}})_n$ converges in the weak*-topology of $M$ to $v^*v$, which is the support projection of~$|x|$.
We deduce that $y^kv^*v=0$.
Using that $v=vv^*v$, we get
\[
y^k
= y^{k-1}|x|v
= y^{k-1}|x|vv^*v
= y^kv^*v
= 0.
\]
By assumption of the induction, we have $y \in [M,I]$, and thus
\[
x 
= \big[ v,|x| \big] + |x|v
\in [M,I],
\]
as desired.
\end{proof}

\begin{rmk}
Let $H$ be a separable, infinite-dimensional Hilbert space, and let~$I$ be a two-sided ideal in the von Neumann algebra $\Bdd(H)$.
Given a nilpotent element $x \in I$, it follows from \autoref{prp:NilpotentVNA} that $x \in [\Bdd(H),I]$, that is, $x$ is a finite sum of commutators of an element in $\Bdd(H)$ and an element in $I$.
Dykema and Krishnaswamy-Usha showed in \cite[Proposition~3.1]{DykKri18NilpotSingleCommutator} that one summand suffices, that is, we have $x = [y,z]$ for some $y \in \Bdd(H)$ and $z \in I$.
Does their result hold for nilpotent elements in two-sided ideals in arbitrary von Neumann algebras?
\end{rmk}

\begin{thm}
\label{prp:IdlVNA}
Let $M$ be a von Neumann algebra with zero commutative summand, and let $I \subseteq M$ be a two-sided ideal.
Then
\[
I = \big[ M,I^{\frac{1}{2}} \big]^2 = \big[ I^{\frac{1}{4}},I^{\frac{1}{4}} \big]^2, 
\]
and
\[
[M,I] 
\subseteq \big[ [M,I^{\frac{1}{2}}], [M,I^{\frac{1}{2}}] \big] 
\subseteq \big[ [I^{\frac{1}{4}},I^{\frac{1}{4}}], [I^{\frac{1}{4}},I^{\frac{1}{4}}] \big] 
\subseteq [I^{\frac{1}{2}},I^{\frac{1}{2}}]. 
\]
\end{thm}
\begin{proof}
We will use that $I$ is a Dixmier ideal, since two-sided ideals in von Neumann algebras are automatically Dixmier ideals (\cite[Proposition~3.4]{GarKitThi23arX:SemiprimeIdls}).
Using \autoref{prp:NilpotentVNA} at the first and last step, and using \autoref{prp:N2-Dixmier} at the second step, we have
\[
[M,I]
= \linSpan_\CC \N_2(I) 
\subseteq \linSpan_\CC \big[ \N_2(I^{\frac{1}{2}}),\N_2(I^{\frac{1}{2}}) \big]
= \big[ [M,I^{\frac{1}{2}}],[M,I^{\frac{1}{2}}] \big].
\]
Using \autoref{prp:AI-I12-I12}, we then get
\[
[M,I]
\subseteq \big[ [M,I^{\frac{1}{2}}],[M,I^{\frac{1}{2}}] \big]
\subseteq \big[ [I^{\frac{1}{4}},I^{\frac{1}{4}}],[I^{\frac{1}{4}},I^{\frac{1}{4}}] \big]
\subseteq \big[ I^{\frac{1}{2}},I^{\frac{1}{2}} \big]
\]

By \autoref{prp:RR0}, there is a unital \stHom{} $M_2(\CC) \oplus M_3(\CC) \to M$. 
We can therefore apply \autoref{prp:I-Dixmier} at the first step, and together with the above we obtain that
\[
I 
= [M,I] + \big[ M,I^{\frac{1}{2}} \big]^2
\subseteq \big[ M,I^{\frac{1}{2}} \big]^2.
\]
Applying \autoref{prp:AI-I12-I12} for the Dixmier ideal $I^{\frac{1}{2}}$, and using \autoref{prp:PowersDixmierIdl} at the last step, we get
\[
I 
\subseteq \big[ M,I^{\frac{1}{2}} \big]^2
\subseteq \big[ I^{\frac{1}{4}}, I^{\frac{1}{4}} \big]^2
\subseteq \big( I^{\frac{1}{2}} \big)^2
= I,
\]
as desired.
\end{proof}

\begin{rmk}
It is likely that the inclusions involving $[M,I]$ in \autoref{prp:IdlVNA} are equalities.
The key question is whether $[I^{\frac{1}{2}},I^{\frac{1}{2}}]$ is a subset of $[M,I]$.
Using that $I$ is a Dixmier ideal and therefore admits a well-behaved theory of roots and powers (\autoref{prp:PowersDixmierIdl}), we have
\[
\big[ I^{\frac{1}{2}},I^{\frac{1}{2}} \big]
= \big[ I^{\frac{1}{4}}I^{\frac{1}{4}},I^{\frac{1}{2}} \big]
\subseteq \big[ I^{\frac{1}{4}},I^{\frac{1}{4}}I^{\frac{1}{2}} \big] + \big[ I^{\frac{1}{2}},I^{\frac{1}{2}}I^{\frac{1}{4}} \big]
= \big[ I^{\frac{1}{4}},I^{\frac{3}{4}} \big],
\]
and by an analogous argument $[I^{\frac{1}{4}},I^{\frac{3}{4}}] \subseteq [I^{\frac{1}{8}},I^{\frac{7}{8}}]$.
Inductively, it follows that $[I^{\frac{1}{2}},I^{\frac{1}{2}}] \subseteq [I^{\frac{1}{2^n}},I^{\frac{2^n-1}{2^n}}]$ for every $n \geq 1$, and thus
\[
\big[ I^{\frac{1}{2}},I^{\frac{1}{2}} \big]
\subseteq \bigcap_{n \geq 1} \big[ I^{\frac{1}{2^n}},I^{\frac{2^n-1}{2^n}} \big]
\subseteq \bigcap_{\varepsilon>0} \big[ M,I^{1-\varepsilon} \big].
\]
It is, however, not clear if $\bigcap_{\varepsilon>0} [M,I^{1-\varepsilon}] = [M,I]$.
\end{rmk}

For two-sided ideals $I,J \subseteq \Bdd(H)$, it was shown in \cite[Theorem~5.10]{DykFigWeiWod04CommStrOpIdls} that $[\Bdd(H),IJ] = [I,J]$.
In particular, we have $[\Bdd(H),I] = [I^{\frac{1}{2}},I^{\frac{1}{2}}]$, and we obtain the following consequence of \autoref{prp:IdlVNA}:

\begin{cor}
Let $H$ be a separable, infinite-dimensional Hilbert space, and let $I$ be a two-sided ideal in $\Bdd(H)$.
Then
\[
[\Bdd(H),I] 
= \big[ [\Bdd(H),I^{\frac{1}{2}}], [\Bdd(H),I^{\frac{1}{2}}] \big] 
= \big[ [I^{\frac{1}{4}},I^{\frac{1}{4}}], [I^{\frac{1}{4}},I^{\frac{1}{4}}] \big] 
= \big[ I^{\frac{1}{2}},I^{\frac{1}{2}} \big]. 
\]
\end{cor}

\begin{qst}
Let $I,J \subseteq M$ be two-sided ideals in a von Neumann algebra.
Do we have $[M,IJ]=[I,J]$?
\end{qst}


\providecommand{\href}[2]{#2}

\end{document}